\documentclass[conference]{IEEEtran}
\IEEEoverridecommandlockouts

\usepackage{cite}
\usepackage{amsmath,amssymb,amsfonts}
\usepackage{algorithmic}
\usepackage{graphicx}
\usepackage{textcomp}
\usepackage{xcolor}
\usepackage{mathtools}
\usepackage{bm}            
\usepackage{mathrsfs}      
\usepackage{amsthm}        

\newtheorem{definition}{Definition}
\newtheorem{theorem}{Theorem}

\newtheorem{lemme}{Lemma}

\newcommand{\R}{\mathbb{R}}
\newcommand{\eqas}[1]{\displaystyle\mathop{ \stackrel{a.s.}{=}}_{#1}}
\newcommand{\geqas}[1]{\displaystyle\mathop{ \stackrel{a.s.}{\geq}}_{#1}}

\newcommand{\subsetas}[1]{\displaystyle\mathop{ \stackrel{a.s.}{\subset}}_{#1}}
\newcommand{\middlesqueezeequ}{\medmuskip=1.5mu \thinmuskip=0.5mu \thickmuskip=2.5mu \nulldelimiterspace=-0.5pt \scriptspace=0pt}
\newcommand{\middlesqueezeequAggressive}{%
  \thinmuskip=0mu
  \medmuskip=0.4mu
  \thickmuskip=0.8mu
  \nulldelimiterspace=-1pt
  \scriptspace=0pt
}
\newcommand{\middlesqueezeequBrutal}{%
  \thinmuskip=0mu
  \medmuskip=0mu
  \thickmuskip=0.3mu
  \nulldelimiterspace=-1.5pt
  \scriptspace=0pt
}

\def\BibTeX{{\rm B\kern-.05em{\sc i\kern-.025em b}\kern-.08em
    T\kern-.1667em\lower.7ex\hbox{E}\kern-.125emX}}

\begin{document}

\title{Equivalence of optimal transport problems to regularization on the family of f-divergences\\
\thanks{This work is supported in part by the European Commission through the H2020-MSCA-RISE-2019 project 872172; the French National Agency for Research (ANR) through the Project ANR-21-CE25-0013 and the project ANR-22-PEFT-0010 of the France 2030 program PEPR Réseaux du Futur; and in part by the Agence de l'innovation de défense (AID) through the project UK-FR 2024352.}
}

\author{
\IEEEauthorblockN{
Maxime Nicaise\IEEEauthorrefmark{1},
Yaiza Bermudez\IEEEauthorrefmark{1},
 and Samir M. Perlaza\IEEEauthorrefmark{1}\IEEEauthorrefmark{2}\IEEEauthorrefmark{4}
}

\IEEEauthorblockA{
Emails: maxime.nicaise@student-cs.fr, \{yaiza.bermudez, samir.perlaza\}@inria.fr
}

\IEEEauthorblockA{\IEEEauthorrefmark{1}%
Centre Inria d'Universit\'{e} Côte d'Azur, INRIA, Sophia Antipolis, France}

\IEEEauthorblockA{\IEEEauthorrefmark{2}%
Laboratoire GAATI, Universit\'{e} de la Polyn\'{e}sie fran\c{c}aise,  Fa`a`\={a}, French Polynesia}

\IEEEauthorblockA{\IEEEauthorrefmark{4}%
ECE Dept. Princeton University, Princeton, 08544 NJ, USA}
}


\maketitle

\begin{abstract}
This work establishes that an optimal transport~(OT) problem regularized by a given $f$-divergence admits the same solution as another OT problem regularized by a different $g$-divergence, under an appropriate transformation of the cost function. This structural equivalence between OT problems regularized by distinct divergences, in the sense of sharing the same unique minimizer, is demonstrated within the framework of Polish spaces with bounded cost functions.
\end{abstract}

\begin{IEEEkeywords}
Optimal Transport, Optimization, Regularization.
\end{IEEEkeywords}
\section{Introduction}
Optimal transport (OT) is established as a fundamental tool for comparing probability distributions, with applications ranging from machine learning to computer graphics \cite{wassersteinautoencoders,wassersteingan,domainadaptation,learningwassersteinloss,computergraphics,IEEEoptimalmachine}. The choice of a suitable discrepancy measure between distributions is essential for tasks such as regression, classification, and generative modeling. In this respect, OT-based distances offer a significant advantage over classical divergences, such as the Kullback--Leibler divergence or other $f$-divergences, as demonstrated in \cite{wassersteingan} and \cite{learningwassersteinloss}. Indeed, the topology induced through Wasserstein distances is weaker and better adapted to comparing measures with non-overlapping supports or those supported on lower-dimensional sets. See \cite{wassersteingan} and \cite{villani}. This property makes OT particularly attractive for modern data analysis. 

Despite these appealing properties, from a computational perspective, solving the Kantorovich formulation of optimal transport involves tackling a large-scale linear program. This task often becomes intractable for high-dimensional or large-sample problems, as shown in \cite{computationaloptitranbsoort}. To mitigate the computational complexity of the original problem and improve the sample complexity of the estimators, regularization is commonly employed. This ranges from classic entropic regularization to structured \cite{domainadaptation} or quadratic regularizers \cite{blondel2018smoothsparseoptimaltransport}. In particular, classic entropic regularization leads to the celebrated Sinkhorn algorithm \cite{cuturi2013sinkhorndistanceslightspeedcomputation}. However, more recent developments highlight the benefits of broadening this framework to general $f$-divergences \cite{dimarino2020optimaltransportlossessinkhorn,terjek}. These schemes yield desirable properties such as sparsity or robustness to outliers. Parallel to these advances in optimal transport, similar variational problems are extensively studied within the field of statistical learning theory, particularly within the statistical empirical risk minimization (ERM)  framework \cite{pacbayes,Perlaza_2024,daunas:hal-04110899,daunas_equival}. Specifically, theoretical equivalence between variational problems regularized with two different $f$-divergences is established in \cite{daunas_equival}, \cite{daunas:hal-04258765} and \cite{daunas_asymmetry}. Therein, it is proved that an ERM problem regularized by a specific $f$-divergence is reformulable using a different divergence and a suitably transformed cost function, yielding the exact same minimizer.

This work establishes an equivalence between optimal transport problems with $f$-divergence regularizations, which is analogous to \cite[Theorem~2]{daunas_equival} in the context of ERM with $f$-divergence regularizations. More specifically, in this work, it is proved that two optimal transport problems regularized by different divergences admit the same unique optimal coupling, up to a suitable transformation of the cost function. This structural equivalence between OT problems is demonstrated within the framework of Polish spaces with bounded cost functions. Under these general assumptions, the existence of optimal potentials for the dual problem and the uniqueness of the optimal coupling for the primal problem are rigorously established. Furthermore, a closed-form expression for the optimal coupling is obtained.

\subsection{Notations and preliminaries}
\textcolor{black}{Given a  function~$f: \mathcal{X} \to \mathcal{Y}$, its effective domain is denoted by  $\operatorname{dom}(f) \stackrel{\triangle}{=} \{x \in \mathcal{X}: f(x) < +\infty\}$; and its range is denoted by $\operatorname{Ran}(f) \stackrel{\triangle}{=} \left\lbrace f(x) : x \in \mathcal{X} \right\rbrace$. Given a set $\mathcal{A}$, its interior is denoted by $\operatorname{int}(\mathcal A)$. 
Given two measures $\mu$ and $\nu$, the notation $\mu \ll \nu$ signifies that $\mu$ is absolutely continuous with respect to $\nu$.} In the following, $\mathcal{X}$ and $\mathcal{Y}$ are Polish spaces.
The set of probability measures over $\mathcal{X}$ equipped with the $\sigma$-field of Borel sets is denoted by $\triangle(\mathcal{X})$.
The set of \textcolor{black}{$p$-integrable functions with respect to $P_X$} is denoted by $\operatorname{L}^p(\mathcal{X},P_X)$.
Given two probability measures $P_X \in \triangle(\mathcal{X})$ and $P_Y \in \triangle(\mathcal{Y})$, the notation \textcolor{black}{$\Pi(P_X,P_Y)$ denotes the set of probability measures on $\mathcal{X}\times \mathcal{Y}$ with marginals $P_X$ and $P_Y$ in $\triangle(\mathcal{X})$ and $\triangle(\mathcal{Y})$, respectively.} The probability measures in $\Pi(P_X,P_Y)$ are called couplings.
The product measure of $P_X \in \triangle(\mathcal{X})$ and $P_Y \in \triangle(\mathcal{Y})$ in $\triangle(\mathcal{X}\times\mathcal{Y})$ is denoted by  $\smash{P_XP_Y}$. Given two functions $\smash{f : \mathcal{X}\to\R}$ and $\smash{g : \mathcal{X}\to\R}$, the notation $f \eqas{P_X} g$ means that $f$ and $g$ are almost surely equal with respect to $P_X$. Similarly, the notation $f \geqas{P_X} g$ means that $f$ is greater than $g$ almost surely with respect to $P_X$. Given $\mathcal A$ and $\mathcal B$, two subsets of $\mathcal{X}$, the notation  $\mathcal A \subsetas{P_X} \mathcal{B}$ means that $P_X(\mathcal A \setminus \mathcal B) = 0$. Let $\mathcal E$ be a Banach space and $\mathcal E'$ be its topological dual space.
Given $x \in \mathcal E $ and $y \in \mathcal E'$, $\langle y,x \rangle$ denotes the dual pairing between $y$ and $x$.
If $f : \mathcal{X} \to \R$ and $g : \mathcal{Y} \to \R$ are two real functions, then $f\oplus g : \mathcal{X}\times\mathcal{Y}\to \R$ is the function defined as $f\oplus g (x,y) = f(x)+g(y)$. Analysis of optimality conditions and the derivation of dual problems rely on the transformation of functions into a dual space representation via the convex (or Fenchel) conjugate $\phi^*$, defined in \cite[Section~2.3]{convexanalysis}. For a function $\phi$ on a Banach space $\mathcal E$ with dual $\mathcal E'$, it is defined for $y \in \mathcal E'$ as $\phi^*(y) = \sup_{x \in \mathcal{E}} (\langle y,x \rangle - \phi(x))$. Constraints are handled through the indicator function $\iota_{\mathcal C}$, which assigns a value of $0$ for elements in the feasible set $\mathcal C$ and $+\infty$ otherwise.
\subsection{Assumptions on $\phi$-divergences}
\label{section_sur_phi_divergences}
The purpose of this section is to define the class of $\phi$-divergences endowed with sufficient regularity to guarantee satisfactory theoretical properties to regularize OT problems.
Let $\phi$ be a convex, proper, lower semi-continuous function from $\R$ to $(-\infty,+\infty]$. Throughout this work, the real numbers $a_\phi$ and $b_\phi$ denote the following quantities
\begin{IEEEeqnarray}{rCl}
  \label{definition_de_a_phi}
  a_\phi &\stackrel{\triangle}{=}& \inf \operatorname{dom}(\phi); \ \text{and} \\
    \label{definition_de_b_phi}
  b_\phi &\stackrel{\triangle}{=}& \sup \operatorname{dom}(\phi).
\end{IEEEeqnarray}
\textcolor{black}{As stated in \cite[Lemma~2.1]{legendretype}, the following limits exist}
\begin{IEEEeqnarray}{rCl}
    \label{definition_de_alpha_phi}
      \alpha_\phi &\stackrel{\triangle}{=}& \lim_{x\to-\infty} \frac{\phi(x)}{x}; \ \text{and} \\
    \label{definition_de_beta_phi}
      \beta_\phi &\stackrel{\triangle}{=}& \lim_{x\to+\infty} \frac{\phi(x)}{x}.
\end{IEEEeqnarray}
Note that if $\operatorname{dom}(\phi) = (0,+\infty)$, then by convention $\alpha_\phi = -\infty.$ \textcolor{black}{It is shown in~ \cite[Lemma~2.1]{legendretype} that} the convex conjugate of $\phi$, which is denoted by $\phi^*$, verifies
\begin{gather}
  \label{domaine_de_phi_etoile}
  \operatorname{int}(\operatorname{dom}(\phi^*)) = (\alpha_\phi, \beta_\phi).
\end{gather}
The class of Legendre functions, characterized by strict convexity and essential smoothness, ensures a smooth one-to-one correspondence between primal and dual variables. The following definition characterizes functions of Legendre type.
\begin{definition}[\cite{legendretype},~Definition~2.5]
  A function $\phi$ is said to be of Legendre type, if it is strictly convex on $\operatorname{dom}(\phi)$; differentiable on $\operatorname{int}(\operatorname{dom}(\phi))$; and its derivative ${\phi}':\operatorname{int}(\operatorname{dom}(\phi))\to\R$ satisfies
\begin{IEEEeqnarray}{rCl}
    \lim_{x\to a_\phi} \phi'(x) &=& -\infty, \text{ if } a_\phi>-\infty;
 \ \text{and} \\
    \lim_{x\to b_\phi} \phi'(x) &=& +\infty, \text{ if } b_\phi<+\infty.
\end{IEEEeqnarray}
\end{definition}
The specific class of Legendre functions possesses certain properties regarding conjugation, defined in \cite[Section~2.3]{convexanalysis}, as detailed in the following lemma.
\begin{lemme}[\cite{legendretype},~Lemma~2.6]
  \label{lemme_proprietes_de_legendre}
  If $\phi : \R \to (-\infty,+\infty]$ is of Legendre type, then so is its convex conjugate $\phi^*$. Moreover $\phi':(a_\phi,b_\phi)\to (\alpha_\phi,\beta_\phi)$ and ${\phi^*} ':(\alpha_\phi,\beta_\phi)\to (a_\phi,b_\phi)$ are continuous, strictly increasing and mutually inverse maps.
\end{lemme}
In the context of this work, the focus is on a specific subset of Legendre functions that are suitable for defining divergences between probability measures, which are defined as follows.
\begin{definition}[Legendre-type generator]
  \label{def_de_P1}
  A Legendre type function $\phi : \R \to (-\infty,+\infty]$ is said to be a Legendre-type generator when $\phi(1) = 0$; $a_\phi=0$ and $b_\phi =+\infty$, with $a_\phi$ in~\eqref{definition_de_a_phi} and $b_\phi$ in~\eqref{definition_de_b_phi}; and its derivative satisfies $\phi'(1) = 0$.
\end{definition}
\noindent
Given a Legendre-type generator $\phi$, it follows from \eqref{domaine_de_phi_etoile} that the interior of the domain of $\phi^*$ is $(-\infty,\beta_\phi)$, and ${\phi^*}' : (-\infty,\beta_\phi) \to (0,+\infty)$ is a one-to-one strictly increasing map, the inverse of which is $\phi' : (0,+\infty) \to (-\infty,\beta_\phi)$.
Note that a Legendre-type generator $\phi$ is always positive, as it is convex and verifies $\phi'(1)= 0$, thus its minimum is achieved at $1$, and the function satisfies $\phi(1) = 0$. Moreover, note that the condition $\phi(1)=\phi'(1)=0$ is not restrictive, as $\phi$-divergences are invariant up to an affine transformation \cite{werner2012fdivergenceconvexbodies}. A comprehensive classification of the common $\phi$-divergences generated by a Legendre-type generator is available in \cite[Tables~1 and 2]{InriaRRnicaise}. The following definition specifies the notion of $\phi$-divergence adopted throughout this work.
\begin{definition}
  \label{definition_phi_divergence}
  Given a Legendre-type generator $\phi$ and two probability measures $P$ and $Q$ defined on the same measurable space, the $\phi$-divergence of $P$ with respect to $Q$ is
  \begin{IEEEeqnarray}{rCl}
    \label{definiton_de_phi_divergence_formule}
    D_\phi(P||Q) = \begin{cases}
      \int \phi(\frac{\mathrm{d}P}{\mathrm{d}Q}(\theta))\mathrm{d}Q(\theta), \  \text{ if } P\ll Q;
\\
      +\infty, \text{ otherwise,}
    \end{cases}
  \end{IEEEeqnarray}
where $\frac{\mathrm{d}P}{\mathrm{d}Q}$ is the Radon-Nikodym derivative of $P$ with respect to $Q$.
\end{definition} 
\noindent
Definition \ref{definition_phi_divergence} is not standard when \textcolor{black}{the function $\phi$} is not assumed to be \textcolor{black}{superlinear (i.e., $\lim_{x\to +\infty} \phi(x)/x = +\infty$)}.
In the typical definition, certain probability measures may yield a finite $\phi$-divergence value even when they are not absolutely continuous with respect to the reference measure;
see~\cite{Csiszar1967, legendretype, AgrawalHorel2020}. However, when a $\phi$-divergence is used as a regularization term in an optimization problem, allowing non–absolutely continuous measures may lead to situations in which the optimizer no longer admits a closed-form expression. This is due to the possible presence of a non-zero singular part, which typically cannot be expressed in closed form, as illustrated in \cite[Theorem~3]{terjek}.
\section{Problem Formulation}
This section establishes the mathematical formulation of optimal transport regularized by a $\phi$-divergence.
The primal problem is introduced, as well as its dual counterpart, and the necessary functional analytic tools.
Throughout this section, the function $\phi$ is a Legendre-type generator (Definition~\ref{def_de_P1}). The measure $P_{XY}$ in $\triangle(\mathcal{X}\times\mathcal{Y})$ is assumed to exhibit given marginal measures denoted by $P_X \in \triangle(\mathcal{X})$ and $P_Y \in \triangle(\mathcal{Y})$. The constant $\lambda$ is a strictly positive real number, and the function
  $\mathsf{c} : \mathcal{X}\times \mathcal{Y} \to \R$ 
is a bounded Borel measurable cost function. These elements constitute the parameters of the optimal transport problem with $\phi$-divergence regularization,  as presented in \cite{dimarino2020optimaltransportlossessinkhorn} and \cite{terjek}.
\subsection{Primal Problem}
The OT problem with $\phi$-divergence regularization is formulated in \cite{dimarino2020optimaltransportlossessinkhorn} and \cite{terjek} as
\begin{IEEEeqnarray}{rCl}
  \label{problem_principal}
  \inf_{P_{XY} \in \Pi(P_X,P_Y)} &&
  \int_{\mathcal{X} \times \mathcal{Y}} \mathsf{c}(x,y)\, \mathrm{d}P_{XY}(x,y)
  \IEEEnonumber \\ && + 
  \lambda\, D_\phi(P_{XY} \,\|\, P_XP_Y).
\end{IEEEeqnarray} 
From Definition~\ref{definition_phi_divergence}, the divergence is infinite whenever $P_{XY}$ is not absolutely continuous with respect to the product of its marginal measures, $P_X$ and $P_Y$.
Consequently, the analysis is restricted to absolutely continuous couplings.
This allows the primal problem to be expressed directly in terms of the Radon-Nikodym derivative of $P_{XY}$ with respect to the product measure $P_X P_Y$, which is a non-negative function in $\operatorname{L}^1(\mathcal{X}\times\mathcal{Y})$. Hence, 
\begin{subequations}
      \label{primal_sur_L1}
\begin{IEEEeqnarray}{rCl}
    \mathscr{P} \stackrel{\triangle }{=}\inf_{r \geq 0}  &&  \middlesqueezeequ\int_{\mathcal{X}\times\mathcal{Y}} \Big( \mathsf{c}(x,y)\, r(x,y) + \lambda \phi\big(r(x,y)\big) \Big) \, \mathrm{d}P_X P_Y \IEEEeqnarraynumspace \\ \label{primal_sur_L1_contrainte_1}
    \text{s.t.} \quad && \int_{\mathcal{Y}} r(\cdot, y) \,\mathrm{d}P_Y(y) \eqas{P_X} 1 \\ \label{primal_sur_L1_contrainte_2}
    && \int_{\mathcal{X}} r(x, \cdot) \,\mathrm{d}P_X(x) \eqas{P_Y} 1,
\end{IEEEeqnarray}
\end{subequations}
where \eqref{primal_sur_L1_contrainte_1} and \eqref{primal_sur_L1_contrainte_2} enforce the marginal constraints on the Radon-Nikodym derivative of $P_{XY}$ with respect to $P_X P_Y$; and  $r$ is a non-negative function in $\operatorname{L}^1(\mathcal{X}\times\mathcal{Y})$.
\subsection{Dual Problem}
The dual problem is derived in \cite[Section~2.2]{InriaRRnicaise} using standard convex analysis tools, and is formulated over the product space $\operatorname{L}^\infty(\mathcal{X}, P_X)\times\operatorname{L}^\infty(\mathcal{Y}, P_Y)$ as follows
\begin{IEEEeqnarray}{rCl}
  \label{dual_sur_L1}
  \middlesqueezeequ
  \nonumber
  \mathscr{D} \stackrel{\triangle }{=}&&\sup_{f,g} \int_{\mathcal{X}} f(x) \,\mathrm{d}P_X(x) + \int_{\mathcal{Y}} g(y) \,\mathrm{d}P_Y(y) \\  && -  \lambda \int_{\mathcal{X}\times\mathcal{Y}} \phi^*\left(\frac{f(x) + g(y) -\mathsf{c}(x,y)}{\lambda}\right)\,\mathrm{d}P_X P_Y, \ \
\end{IEEEeqnarray}
where the function $\phi^*$ is the convex conjugate of the function $\phi$, as defined in \cite[Section~2.3]{convexanalysis}.
\section{Existence and optimality properties}
\subsection{Attainment of the dual problem}
\label{section_attainment_dual}
The existence of optimal potentials for the dual problem in \eqref{dual_sur_L1} is established by extending the results of \cite{dimarino2020optimaltransportlossessinkhorn} and \cite{terjek} to Polish spaces and bounded cost functions. This result follows from constructing maximizing sequences through alternating coordinate-wise updates \cite[Section~2.4.1]{InriaRRnicaise}.
\begin{theorem}
  \label{theoreme_existence_de_potentiels}
  The supremum in problem (\ref{dual_sur_L1}) is attained for a unique pair of potentials $(f_0,g_0)$ up to the transformation $(f,g)\mapsto (f+a,g-a)$, for $a \in\R$.
\end{theorem}
\begin{IEEEproof}
  The proof is presented in \cite{InriaRRnicaise}.
\end{IEEEproof}
The uniqueness is stated ``up to a constant" because the dual functional depends on the potentials primarily through their sum $f(x)+g(y)$.
Consequently, adding a scalar $a$ to $f$ and subtracting it from $g$ leaves the objective value unchanged.
\subsection{Characterization of dual optimizers}
The characterization of optimal dual potentials relies on first-order optimality conditions obtained by differentiating the dual objective in \eqref{dual_sur_L1}. This requires differentiating the integral term involving the convex conjugate $\phi^*$.
Since the derivative ${\phi^*}'$ blows up as its argument approaches the boundary $\beta_\phi$ (Lemma~\ref{lemme_proprietes_de_legendre}), differentiation under the integral sign is only justified when the argument of $\phi^*$ remains uniformly bounded away from this singularity. To ensure this regularity, we restrict the analysis to a subset of admissible potentials on which the dual objective is differentiable.
Specifically, define the set
\begin{IEEEeqnarray}{rCl}
  \label{definition_A_phi_c}
  \nonumber
   \mathcal{A}_{\phi,\mathsf{c}} \stackrel{\triangle}{=} \{(f,g) \in \operatorname{L}^\infty(\mathcal{X},P_X)\times \operatorname{L}^\infty(\mathcal{Y},P_Y): \hspace{1cm} \\  \exists \gamma<\beta_\phi, \ \operatorname{Ran}\left(\frac{f\oplus g -\mathsf{c}}{\lambda}\right) \subsetas{P_XP_Y} (-\infty,\gamma)\}.
\end{IEEEeqnarray}
Within this set, the first-order optimality conditions are well-defined and admit a simple characterization, as stated in the following lemma.
\begin{lemme}
  \label{lemme_solutions_ssi_c_transform_differentiable}
  A pair of functions $(\hat{f},\hat{g}) \in \mathcal{A}_{\phi,\mathsf{c}}$, where $\mathcal{A}_{\phi,\mathsf{c}}$ is defined in \eqref{definition_A_phi_c}, are maximizing potentials for the dual problem in (\ref{dual_sur_L1}) if and only if:
  \begin{IEEEeqnarray}{rCl}
    \label{systeme_critere_condition_pour_etre_optimaux}
      \int_\mathcal{Y} {\phi'}^{-1}\left(\frac{\hat{f}(x)+\hat{g}(y) - \mathsf{c}(x,y)}{\lambda}\right)\mathrm{d}P_Y(y) &\eqas{P_X}& 1 ;
 \ \text{and} \ \  \ \ \ \ \\ \label{systeme_critere_condition_pour_etre_optimaux_2}
      \int_\mathcal{X} {\phi'}^{-1}\left(\frac{\hat{f}(x)+\hat{g}(y) - \mathsf{c}(x,y)}{\lambda}\right)\mathrm{d}P_X(x) &\eqas{P_Y}& 1.
  \end{IEEEeqnarray}
\end{lemme}
\begin{IEEEproof}
  The proof is presented in \cite{InriaRRnicaise}.
\end{IEEEproof}
\subsection{Attainment of the primal problem}
As shown in \cite[Theorem~3]{terjek}, an optimal solution to \eqref{problem_principal} may not exist unless Definition~\ref{definition_phi_divergence} is extended to assign finite $\phi$-divergence to measures admitting a non-zero singular part with respect to $P_X P_Y$. Under this extension, the optimal coupling itself may contain a singular component, which poses significant practical difficulties: this component has no closed-form representation and precludes the use of the Sinkhorn algorithm. Consequently, additional assumptions are required to guarantee the existence of a computable solution, as shown hereunder.

\begin{theorem}
  \label{theoreme_pb_primal}
  Suppose that there exists a pair of optimal potentials $(\hat{f},\hat{g})$ in $\mathcal{A}_{\phi,\mathsf{c}}$ for the dual problem in (\ref{dual_sur_L1}).
Then, strong duality holds between the optimization problems (\ref{primal_sur_L1}) and (\ref{dual_sur_L1}), and the unique minimizer of the primal problem satisfies
  \begin{IEEEeqnarray}{rCl}
    \frac{\mathrm{d}P_{XY}^{(\mathsf{c},\phi)}}{\mathrm{d}P_XP_Y}(x,y) \eqas{P_XP_Y} {\phi^*}'\left(\frac{\hat{f}(x) + \hat{g}(y)-\mathsf{c}(x,y)}{\lambda}\right).
\end{IEEEeqnarray}
\begin{IEEEproof}
  The proof is presented in \cite{InriaRRnicaise}.
\end{IEEEproof}
\end{theorem}
\subsection{Sufficient conditions for primal attainment}
\label{section_suficient_conditions}
The validity of Theorem \ref{theoreme_pb_primal} relies on the assumption that there exists a pair of optimal potentials in $\mathcal{A}_{\phi,\mathsf{c}}$.
Intuitively, this condition requires that the optimal potentials stay uniformly bounded away from the singularity of the dual function ${\phi^*}'$.
While verifying this property directly for general Polish spaces can be challenging, it holds automatically in several significant scenarios commonly encountered in applications.
This section explores sufficient conditions ensuring this requirement is met.

The first sufficient condition concerns divergences with superlinear growth, such as the Kullback-Leibler divergence.
In this setting, the domain of the conjugate function $\phi^*$ extends to the entire real line, effectively removing the boundary singularity.
\begin{lemme}
  \label{lemme_superlineaire}
  If $\beta_\phi$ in \eqref{definition_de_beta_phi} is equal to $+\infty$, any pair of optimal potentials is in $\mathcal{A}_{\phi,\mathsf{c}}$.
\end{lemme}
\begin{IEEEproof}
  The proof is presented in \cite{InriaRRnicaise}.
\end{IEEEproof}
The second sufficient condition addresses the case of discrete OT, where the underlying spaces are finite sets.
Here, the finiteness of the space prevents the values of the potentials from accumulating asymptotically against the boundary $\beta_\phi$.
\begin{lemme}
  \label{condition_suffisante_sets_finis}
  If the sets $\mathcal{X}$ and $\mathcal{Y}$ are both finite, then any pair of optimal potentials is in $\mathcal{A}_{\phi,\mathsf{c}}$.
\end{lemme}
\begin{IEEEproof}
  The proof is presented in \cite{InriaRRnicaise}.
\end{IEEEproof}
\section{Equivalence of the problem to regularization on the family of $\phi$-divergences}
\subsection{Main results}
The following theorem constitutes the main result of this work. Intuitively, it states that the choice of divergence in regularized OT is not intrinsic to the problem: any effect of the regularization on the optimal coupling can equivalently be attributed to a modification of the transport cost. The result and its proof have been adapted from \cite[Theorem~2]{daunas_equival}, which develops in the context of empirical risk minimization with $\phi$-divergence regularization. Although there are fundamental technical differences due to the presence of additional constraints in the case of optimal transport, the two results convey the same meaning in their corresponding scenarios.
\begin{theorem}
\label{main_result}
Assume that the problem in \eqref{dual_sur_L1} admits a pair of optimal potentials
\((\hat f,\hat g)\in\mathcal{A}_{\phi,\mathsf c}\),
and let \(\psi\) be a Legendre-type generator (Definition~\ref{def_de_P1}).
Then, there exists a bounded function
\(\mathsf v:\mathcal X\times\mathcal Y\to\mathbb R\)
such that the problem in \eqref{problem_principal} is equivalent,
in the sense that both problems admit the same minimizer,
to the following optimization problem:
\begin{IEEEeqnarray}{rCl}
\label{problem_secondaire}
\nonumber
\inf_{P_{XY}\in\Pi(P_X,P_Y)}
\int_{\mathcal X\times\mathcal Y}
\mathsf v(x,y)\,\mathrm dP_{XY}(x,y) \\ 
+ \lambda D_\psi\!\left(P_{XY}\,\middle\|\,P_XP_Y\right).
\end{IEEEeqnarray}
\end{theorem}
\begin{IEEEproof}
  Let $(\hat{f},\hat{g}) \in \mathcal{A}_{\phi,\mathsf{c}}$ be maximizing potentials for the dual problem in (\ref{dual_sur_L1}).
Define the cost function $\mathsf{v}$ as
  \begin{IEEEeqnarray}{rCl}
    \middlesqueezeequ
    \label{definition_de_v}
    \mathsf{v}(x,y) = -\lambda{\psi'}\left({\phi'}^{-1}\left(\frac{\hat{f}(x) + \hat{g}(y) -\mathsf{c}(x,y)}{\lambda}\right)\right) + 1.
  \end{IEEEeqnarray}
  Given that the functions $\mathsf{c}$, $\hat{f}$ and $\hat{g}$ are bounded and that the range of $(x,y)\mapsto \frac{\hat{f}(x) + \hat{g}(y) -\mathsf{c}(x,y)}{\lambda}$ is bounded and contained in some interval of the form $(-\infty,\gamma)$, with $\gamma<\beta_\phi$, the range of $(x,y)\mapsto {\phi'}^{-1}\left(\frac{\hat{f}(x) + \hat{g}(y) -\mathsf{c}(x,y)}{\lambda}\right)$ is bounded and contained in an interval of the form $(\rho,+\infty)$ with $\rho>0$.
Since $\psi'$ is $C^1$ on $(0,\infty)$, the function
  \begin{IEEEeqnarray}{rCl}
    (x,y)\mapsto {\psi'}\left({\phi'}^{-1}\left(\frac{\hat{f}(x) + \hat{g}(y) -\mathsf{c}(x,y)}{\lambda}\right)\right)
  \end{IEEEeqnarray}
   is also bounded.
Thus, the function $\mathsf{v}$ is bounded, ensuring that problem (\ref{problem_secondaire}) fits within the setting established in previous sections.
Let $\check{f} \in \operatorname{L}^\infty(\mathcal{X},P_X)$ and $\check{g}\in \operatorname{L}^\infty(\mathcal{Y},P_Y)$ be defined by
  \begin{IEEEeqnarray}{rCl}
    \check{f}(x) = \frac{1}{2}, \ \forall x\in \mathcal{X};  & \mbox{ and } &
      \check{g}(y) = \frac{1}{2}, \ \forall y\in \mathcal{Y}.
\end{IEEEeqnarray}
  The proof proceeds by demonstrating that the functions $\check{f},\check{g}$ satisfy (\ref{systeme_critere_condition_pour_etre_optimaux}) and \eqref{systeme_critere_condition_pour_etre_optimaux_2} for the problem in \eqref{problem_secondaire}.
Verification proceeds as follows: let $x\in \mathcal{X}$ be fixed, and compute the integral
  \begin{IEEEeqnarray}{rCl}
    \nonumber
    &&\int_\mathcal{Y} {\psi'}^{-1}\left(\frac{\check{f}(x)+\check{g}(y)-\mathsf{v}(x,y)}{\lambda}\right)\mathrm{d}P_Y(y) \\ \middlesqueezeequ  &=& \middlesqueezeequAggressive \int_\mathcal{Y} {\psi'}^{-1}\left(\frac{1+\lambda{\psi'}({\phi'}^{-1}(\frac{\hat{f}(x) + \hat{g}(y) -\mathsf{c}(x,y)}{\lambda})) - 1}{\lambda}\right)\mathrm{d}P_Y(y) \ \ \ \ \  \\ \label{equation_derniere_preuve_1} 
    &=& \int_\mathcal{Y} {\phi'}^{-1}\left(\frac{\hat{f}(x) + \hat{g}(y) -\mathsf{c}(x,y)}{\lambda}\right)\mathrm{d}P_Y(y).
  \end{IEEEeqnarray}
  Since $(\hat{f},\hat{g}) \in \mathcal{A}_{\phi,\mathsf{c}}$ is the solution to (\ref{dual_sur_L1}), Lemma \ref{lemme_solutions_ssi_c_transform_differentiable} implies
  \begin{IEEEeqnarray}{rCl}
    \label{equation_derniere_preuve_2}
        \int_\mathcal{Y} {\phi'}^{-1}\left(\frac{\hat{f}(x) + \hat{g}(y) -\mathsf{c}(x,y)}{\lambda}\right)\mathrm{d}P_Y(y) \eqas{P_X} 1.
  \end{IEEEeqnarray}
  Combining \eqref{equation_derniere_preuve_1} and \eqref{equation_derniere_preuve_2} yields
  \begin{IEEEeqnarray}{rCl}
    \label{systeme_dans_preuve_pour_psi}
      \int_\mathcal{Y} {\psi'}^{-1}\left(\frac{\check{f}(x)+\check{g}(y) - \mathsf{v}(x,y)}{\lambda}\right)\mathrm{d}P_Y(y) &\eqas{P_X}& 1.
  \end{IEEEeqnarray}
  Moreover, applying the same reasoning for a fixed $y\in \mathcal{Y}$ yields
  \begin{IEEEeqnarray}{rCl}
  \label{systeme_dans_preuve_pour_psi_2}
      \int_{ \mathcal X} {\psi'}^{-1}\left(\frac{\check{f}(x)+\check{g}(y) - \mathsf{v}(x,y)}{\lambda}\right)\mathrm{d}P_X(x) &\eqas{P_Y}& 1.
  \end{IEEEeqnarray}
 From Lemma \ref{lemme_solutions_ssi_c_transform_differentiable} $(\hat{f},\hat{g})$ are optimal potentials for the dual of \eqref{problem_secondaire}, which in turn requires verifying that $(\check{f},\check{g}) \in \mathcal{A}_{\psi,\mathsf{v}}$.
This set is defined analogously to $\mathcal{A}_{\phi,\mathsf{c}}$ in \eqref{definition_A_phi_c}, by replacing the parameters $\phi$ and $\mathsf{c}$ with $\psi$ and $\mathsf{v}$ respectively.
This condition requires that the range of $l:(x,y)\mapsto \frac{\check{f}(x)+\check{g}(y) - \mathsf{v}(x,y)}{\lambda}$ satisfies $\operatorname{Ran}(l)\subsetas{P_XP_Y} (-\infty,\gamma)$ for some $\gamma<\beta_\psi$.
Let $(x,y) \in \mathcal{X}\times \mathcal{Y}$ be fixed, and consider
  \begin{IEEEeqnarray}{rCl}
    \label{equation_derniere_preuve_3}
    \middlesqueezeequBrutal
    \frac{\check{f}(x)+\check{g}(y) - \mathsf{v}(x,y)}{\lambda} \ = \psi'\left( {\phi'}^{-1}\left(\frac{\hat{f}(x) + \hat{g}(y) -\mathsf{c}(x,y)}{\lambda}\right)\right). \ \ \
\end{IEEEeqnarray}
  Since the pair $(\hat{f},\hat{g})$ is in the set $\mathcal{A}_{\phi,\mathsf{c}}$ defined in \eqref{definition_A_phi_c}, and due to continuity and strict monotonicity of ${\phi'}^{-1}$ on $(-\infty,\beta_\phi)$ (see Section \ref{section_sur_phi_divergences}), the range of the function $(x,y)\mapsto {\phi'}^{-1}\left(\frac{\hat{f}(x) + \hat{g}(y) -\mathsf{c}(x,y)}{\lambda}\right)$ is bounded and satisfies 
  \begin{IEEEeqnarray}{rCl}
    \middlesqueezeequ
    \operatorname{Ran}\left( {\phi'}^{-1}\left(\frac{\hat{f}(x) + \hat{g}(y) -\mathsf{c}(x,y)}{\lambda}\right)\right) \subsetas{P_XP_Y} (\rho,+\infty),
  \end{IEEEeqnarray}
   for some $\rho>0.$ Given that $\psi'$ is continuous and strictly increasing on $(0,+\infty)$, the range of $(x,y)\mapsto \psi'\left( {\phi'}^{-1}\left(\frac{\hat{f}(x) + \hat{g}(y) -\mathsf{c}(x,y)}{\lambda}\right)\right)$ satisfies 
  \begin{IEEEeqnarray}{rCl}
    \middlesqueezeequAggressive
    \operatorname{Ran}\left( \psi'\left( {\phi'}^{-1}\left(\frac{\hat{f}(x) + \hat{g}(y) -\mathsf{c}(x,y)}{\lambda}\right)\right)\right) \subsetas{P_XP_Y}(-\infty,\gamma), \ \ \ \ \
  \end{IEEEeqnarray}
   for some $\gamma<\beta_\psi$.
Therefore, by \eqref{equation_derniere_preuve_3}, $(\check{f},\check{g})$ belongs to $\mathcal{A}_{\psi,\mathsf{v}}$ and satisfies (\ref{systeme_dans_preuve_pour_psi}) and (\ref{systeme_dans_preuve_pour_psi_2}).
Lemma \ref{lemme_solutions_ssi_c_transform_differentiable} implies that $(\check{f},\check{g})$ are maximizing potentials for the dual problem of (\ref{problem_secondaire}).
\newline
  \newline
  Finally, Theorem \ref{theoreme_pb_primal} gives the form of the solution of the problem in (\ref{problem_secondaire}), denoted by $P_{XY}^{(\mathsf{v},\psi)}$ as
  \begin{IEEEeqnarray}{rCl}
& &  \frac{\mathrm{d}P_{XY}^{(\mathsf{v},\psi)}}{\mathrm{d}P_XP_Y}(x,y)  = {\psi'}^{-1}\left(\frac{\check{f}(x)+\check{g}(y)-v(x,y)}{\lambda}\right) \\ & =& \middlesqueezeequBrutal   {\psi'}^{-1}\left(\psi'\left( {\phi'}^{-1}\left(\frac{\hat{f}(x) + \hat{g}(y) -\mathsf{c}(x,y)}{\lambda}\right)\right) \right) \\ & =& {\phi'}^{-1}\left(\frac{\hat{f}(x) + \hat{g}(y) -\mathsf{c}(x,y)}{\lambda}\right)  =\frac{\mathrm{d}P_{XY}^{(\mathsf{c},\phi)}}{\mathrm{d}P_XP_Y}(x,y), 
  \end{IEEEeqnarray}
where $P_{XY}^{(\mathsf{c},\phi)}$ denotes the solution to the problem in \eqref{primal_sur_L1}. Thus, the two problems share the same minimizer, which completes the proof.
\end{IEEEproof}

\subsection{Examples}

This section illustrates the structural equivalence established in Theorem~\ref{main_result} by considering fixed Polish spaces $\mathcal{X}$ and $\mathcal{Y}$, a fixed bounded cost function $\mathsf{c}$, and a fixed regularization parameter $\lambda>0$. Let $\phi$ denote the generator of the Kullback--Leibler (KL) divergence, defined by $\phi(x)=x\log(x)-x+1$. For this KL generator, the inverse of its derivative is given by
\begin{IEEEeqnarray}{rCl}
\label{derivee_de_KL}
{\phi'}^{-1}(y) = \exp(y).
\end{IEEEeqnarray}
Moreover, the quantity $\beta_\phi$ in~\eqref{definition_de_beta_phi} is equal to $+\infty$. Therefore, it follows from Lemma~\ref{lemme_superlineaire} that the corresponding optimal potentials for the dual problem in~\eqref{dual_sur_L1}, denoted by $\hat{f}$ and $\hat{g}$, belong to the set $\mathcal{A}_{\phi,\mathsf{c}}$  in~\eqref{definition_A_phi_c}.
As a consequence, all the assumptions of Theorem~\ref{main_result} are satisfied for any target Legendre-type generator $\psi$ satisfying Definition~\ref{def_de_P1}. In particular, for the transformed cost function defined in \eqref{definition_de_v}, the problems in \eqref{problem_principal} and \eqref{problem_secondaire} admit the same unique minimizer. The following examples illustrate the resulting transformed cost function~$\mathsf{v}$ obtained via~\eqref{definition_de_v} for various choices of target regularization.
\subsubsection{From KL to reverse KL regularization}
Consider the target regularization defined by the reverse KL generator $\psi(x) = -\log(x) + x - 1$, whose derivative is $\psi'(x) = 1 - \frac{1}{x}$. From~\eqref{definition_de_v} and \eqref{derivee_de_KL}, the transformed cost function $\mathsf{v}$ ensuring that the initial problem~\eqref{problem_principal} and the transformed problem~\eqref{problem_secondaire} share the same unique minimizer is given by
\begin{IEEEeqnarray}{rCl}
  \middlesqueezeequAggressive \mathsf{v}(x,y) & = & -\lambda \psi'\left( \exp\left( \frac{\hat{f}(x) + \hat{g}(y) - \mathsf{c}(x,y)}{\lambda} \right) \right) + 1  \middlesqueezeequAggressive \\ 
& = & -\lambda \left( 1 - \exp\left( - \ \frac{\hat{f}(x) + \hat{g}(y) - \mathsf{c}(x,y)}{\lambda} \right) \right) + 1 \middlesqueezeequBrutal \\
& = & \lambda \exp\left( \frac{\mathsf{c}(x,y) - \hat{f}(x) - \hat{g}(y)}{\lambda} \right) - \lambda + 1.
\end{IEEEeqnarray}

\subsubsection{From KL to Jensen-Shannon regularization}
In the case of a transition to Jensen--Shannon regularization, the target generator is $\psi(x) = x\log(x) - (x+1)\log\left(\frac{x+1}{2}\right)$, with derivative $\psi'(x) = \log\left(\frac{2x}{x+1}\right)$. The transformed cost function $\mathsf{v}$ preserving the minimizer in that case is expressed as
\begin{IEEEeqnarray}{rCl}
\mathsf{v}(x,y) & = & -\lambda \log\left( \frac{2\exp\left( \frac{\hat{f}(x) + \hat{g}(y) - \mathsf{c}(x,y)}{\lambda} \right)}{\exp\left( \frac{\hat{f}(x) + \hat{g}(y) - \mathsf{c}(x,y)}{\lambda} \right) + 1} \right) + 1  \\
& = & \lambda \log\left( \frac{\exp\left( \frac{\hat{f}(x) + \hat{g}(y) - \mathsf{c}(x,y)}{\lambda} \right) + 1}{2\exp\left( \frac{\hat{f}(x) + \hat{g}(y) - \mathsf{c}(x,y)}{\lambda} \right)} \right) + 1.
\end{IEEEeqnarray}
By simplifying the argument within the logarithm, the relation becomes
\begin{IEEEeqnarray}{rCl}
\mathsf{v}(x,y) & = & \lambda \log\left( \frac{1 + \exp\left( \frac{\mathsf{c}(x,y) - \hat{f}(x) - \hat{g}(y)}{\lambda} \right)}{2} \right) + 1.
\end{IEEEeqnarray}
These examples confirm that changing the regularizing divergence manifests as a non-linear adjustment of the cost function, ensuring that the unique minimizer of the regularized optimal transport problem remains unchanged.
\section{Conclusions}
This work investigates $\phi$-divergence regularized optimal transport in Polish spaces with bounded cost functions. The existence of optimal dual potentials and unique primal couplings is established, providing closed-form expressions for the latter. The central contribution, Theorem \ref{main_result}, demonstrates a structural equivalence: a change in the regularization divergence can be compensated by a non-linear transformation of the cost function while preserving the same unique minimizer. This result remains primarily theoretical, since the transformed cost $\mathsf{v}$ in~\eqref{definition_de_v} explicitly depends on the dual potentials $(\hat{f},\hat{g})$ of the original problem. Thus, it does not reduce the computational complexity of solving regularized OT problems. Nonetheless, it provides a unifying framework for regularized OT and offers new perspectives for the characterization and sampling of optimal couplings.

\IEEEtriggeratref{13}
\bibliographystyle{IEEEtran}
\bibliography{biblio2}

@incollection{werner2012fdivergenceconvexbodies,
author={Werner, Elisabeth M.},
title={f-Divergence for Convex Bodies},
bookTitle={Asymptotic Geometric Analysis: {P}roceedings of the Fall 2010 Fields Institute Thematic Program},
year={2013},
publisher={Springer},
address={New York, NY},
pages="381--395"}

@article{Csiszar1967,
  author={Csisz{\'a}r, Imre},
  journal={Studia Scientiarum Mathematicarum Hungarica},
  title={Information-type measures of difference of probability distributions and indirect observations},
  year={1967},
  volume={2},
  pages={299--318}
}

@article{AgrawalHorel2020,
  author={Agrawal, Rajat and Horel, Thibaut},
  journal={Journal of Machine Learning Research},
  title={Optimal Bounds Between f-Divergences and Integral Probability Metrics},
  year={2021},
  volume={22},
  number={1},
  pages={5662--5720},
  month={Jan.}
}

@article{legendretype,
  author={Borwein, Jonathan M. and Lewis, Adrian S.},
  journal={SIAM Journal on Optimization},
  title={Partially-Finite Programming in $L_1$ and the Existence of Maximum Entropy Estimates},
  year={1993},
  volume={3},
  number={2},
  pages={248--267},
  month={May}
}

@inproceedings{wassersteingan,
  author={Arjovsky, Martin and Chintala, Soumith and Bottou, L{\'e}on},
  booktitle={Proceedings of the International Conference on Machine Learning (ICML)},
  title={Wasserstein {GAN}},
  year={2017},
  pages={214--223},
  month={Jul.}
}

@article{domainadaptation,
  author={Courty, Nicolas and Flamary, R{\'e}mi and Tuia, Devis and Rakotomamonjy, Alain},
  journal={IEEE Transactions on Pattern Analysis and Machine Intelligence},
  title={Optimal Transport for Domain Adaptation},
  year={2017},
  volume={39},
  number={9},
  pages={1852--1865},
  month={Sep.}
}

@inproceedings{wassersteinautoencoders,
  author={Tolstikhin, Ilya and Bousquet, Olivier and Gelly, Sylvain and Sch{\"o}lkopf, Bernhard},
  booktitle={Proceedings of the International Conference on Learning Representations (ICLR)},
  title={Wasserstein Auto-Encoders},
  year={2018},
  month={May}
}

@inproceedings{learningwassersteinloss,
author = {Frogner, Charlie and Zhang, Chiyuan and Mobahi, Hossein and Araya-Polo, Mauricio and Poggio, Tomaso},
title = {Learning with a {W}asserstein loss},
year = {2015},
month = {Dec.},
booktitle = {Proceedings of the 29th International Conference on Neural Information Processing Systems - Volume 2},
pages = {2053--2061},
numpages = {9}
}

@article{computergraphics,
  author={Solomon, Justin and de Goes, Fernando and Peyr{\'e}, Gabriel and Cuturi, Marco and Butscher, Adrian and Nguyen, Andy and Du, Tao and Guibas, Leonidas},
  journal={ACM Transactions on Graphics (TOG)},
  title={Convolutional {W}asserstein distances: {E}fficient optimal transportation on geometric domains},
  year={2015},
  volume={34},
  number={4},
  pages={1--11},
  month={Jul.}
}

@inproceedings{blondel2018smoothsparseoptimaltransport,
  author={Blondel, Mathieu and Seguy, Vivien and Rolet, Antoine},
  booktitle={Proceedings of the International Conference on Artificial Intelligence and Statistics (AISTATS)},
  title={Smooth and Sparse Optimal Transport},
  year={2018},
  pages={314--323},
  month={Mar.}
}

@techreport{daunas:hal-04110899,
  author={Daunas, Francisco and Esnaola, I{\~n}aki and Perlaza, Samir M. and Poor, H. Vincent},
  title={Empirical Risk Minimization with Relative Entropy Regularization {T}ype-{II}},
  year={2023},
  number={RR-9508},
  institution={INRIA, Centre Inria d'Universit{\'e} C{\^o}te d'Azur},
  address={Sophia Antipolis, France},
  month={May}
}

@book{pacbayes,
  author={Catoni, Olivier},
  title={{PAC-B}ayesian Supervised Classification: The Thermodynamics of Statistical Learning},
  year={2007},
  publisher={Institute of Mathematical Statistics},
  address={Beachwood, OH, USA},
  volume={56},
  edition={first}
}

@techreport{daunas:hal-04258765,
  author={Daunas, Francisco and Esnaola, I{\~n}aki and Perlaza, Samir M. and Poor, H. Vincent},
  title={Empirical Risk Minimization with f-Divergence Regularization in Statistical Learning},
  year={2023},
  number={RR-9521},
  institution={INRIA, Centre Inria d'Universit{\'e} C{\^o}te d'Azur},
  address={Sophia Antipolis, France},
  month={Oct.}
}

@article{Perlaza_2024,
  author={Perlaza, Samir M. and Bisson, Gaetan and Esnaola, I{\~n}aki and Jean-Marie, Alain and Rini, Stefano},
  journal={IEEE Transactions on Information Theory},
  title={Empirical Risk Minimization With Relative Entropy Regularization},
  year={2024},
  volume={70},
  number={7},
  pages={5122--5161},
  month={Jul.}
}

@inproceedings{daunas_equival,
  author={Daunas, Francisco and Esnaola, I{\~n}aki and Perlaza, Samir M. and Poor, H. Vincent},
  booktitle={Proceedings of the IEEE International Symposium on Information Theory (ISIT)},
  title={Equivalence of Empirical Risk Minimization to Regularization on the Family of $f-$Divergences},
  year={2024},
  pages={759--764},
  month={Jul.}
}

@book{computationaloptitranbsoort,
  author={Peyr{\'e}, Gabriel and Cuturi, Marco},
  title={Computational Optimal Transport},
  year={2020},
  publisher={Foundations and Trends in Machine Learning},
  address={Hanover, MA, USA},
  edition={first}
}

@article{IEEEoptimalmachine,
  author={Pereira, Luiz Manella and Amini, M. Hadi},
  journal={IEEE Access},
  title={A Survey on Optimal Transport for Machine Learning: Theory and Applications},
  year={2025},
  volume={13},
  pages={26506--26526},
  month={Jan.}
}

@article{daunas_asymmetry,
  author={Daunas, Francisco and Esnaola, I{\~n}aki and Perlaza, Samir M. and Poor, H. Vincent},
  journal={IEEE Transactions on Information Theory},
  title={Asymmetry of the Relative Entropy in the Regularization of Empirical Risk Minimization},
  year={2025},
  volume={71},
  number={8},
  pages={6198--6226},
  month={Aug.}
}

@inproceedings{cuturi2013sinkhorndistanceslightspeedcomputation,
  author={Cuturi, Marco},
  booktitle={Proceedings of the International Conference on Neural Information Processing Systems (NeurIPS)},
  title={Sinkhorn Distances: Lightspeed Computation of Optimal Transportation Distances},
  year={2013},
  volume={2},
  pages={2292--2300},
  month={Dec.}
}

@book{convexanalysis,
  author={Zalinescu, Constantin},
  title={Convex analysis in general vector spaces},
  year={2002},
  publisher={World Scientific},
  address={Singapore},
  edition={first}
}

@book{villani,
  author={Villani, C{\'e}dric},
  title={Optimal Transport: Old and New},
  year={2009},
  publisher={Springer-Verlag},
  address={Berlin, Heidelberg},
  edition={first}
}

@inproceedings{terjek,
  author={Terj{\'e}k, D{\'a}vid and Gonz{\'a}lez-S{\'a}nchez, Diego},
  booktitle={Proceedings of the International Conference on Artificial Intelligence and Statistics (AISTATS)},
  title={Optimal transport with f-divergence regularization and generalized Sinkhorn algorithm},
  year={2022},
  volume={151},
  pages={5135--5165},
  month={Mar.}
}

@article{dimarino2020optimaltransportlossessinkhorn,
  author={Di Marino, Simone and Gerolin, Augusto},
  journal={arXiv preprint arXiv:2007.00976},
  title={Optimal Transport losses and Sinkhorn algorithm with general convex regularization},
  year={2020},
  month={Jul.}
}

@techreport{InriaRRnicaise,
  author = {Maxime Nicaise and Yaiza Bermudez and Samir M. Perlaza},
  title = {On Optimal Transport with f-Divergence Regularization},
  year = {2026},
  number = {RR-9607},
  institution = {INRIA, Centre Inria d'Université Côte d'Azur},
  address = {Sophia Antipolis, France},
  month = {January}
}
\end{document}